\documentclass[10pt,leqno,final]{amsart}
\usepackage[colorlinks,linkcolor=blue,citecolor=blue,urlcolor=blue]{hyperref}
\usepackage{amsfonts,amssymb}
\usepackage{color}
\usepackage{graphicx}
\usepackage[notref,notcite]{showkeys}
\usepackage{mathtools}
\usepackage{extarrows}
\usepackage{algorithmic}   
\usepackage{caption}
\usepackage{subcaption}
\usepackage{tikz}
\usepackage{mathrsfs}
\usepackage{enumitem}

\allowdisplaybreaks

\newtheorem{theorem}{Theorem}[section]
\newtheorem{lemma}[theorem]{Lemma}
\newtheorem{corollary}[theorem]{Corollary}
\newtheorem{proposition}[theorem]{Proposition}
\newtheorem{assumption}[theorem]{Assumption}

\theoremstyle{definition}

\newtheorem*{example*}{Example}
\newtheorem{remark}[theorem]{Remark}

\numberwithin{equation}{section}

\usepackage{my_latex_commands}

\newcommand*{\Borel}{\mathfrak{B}}
\newcommand*{\Measure}{\mathfrak{M}}
\newcommand{\cbound}{\underbar{\textup{c}}}
\newcommand{\Prob}{\mathfrak{P}}

\newcommand{\cost}{c_{\textup{d}}}
\newcommand{\marg}{\mu_2^{\textup{d}}}

\begin{document}

\title[Bilevel Kantorovich Problem, Part II]{Bilevel Optimization of the 
Kantorovich Problem and its Quadratic Regularization\\
Part II: Convergence Analysis}
\thanks{This research was supported by the German Research Foundation (DFG) under grant 
number~LO 1436/9-1 within the priority program Non-smooth and Complementarity-based
Distributed Parameter Systems: Simulation and Hierarchical Optimization (SPP~1962).}

\author{Sebastian Hillbrecht} \address{ Sebastian Hillbrecht, Technische Universit\"at Dortmund, Fakult\"at f\"ur
  Mathematik, Lehrstuhl X, Vogelpothsweg 87, 44227 Dortmund, Germany}
\email{sebastian.hillbrecht@tu-dortmund.de}

\author{Paul Manns} \address{Paul Manns, Technische Universit\"at Dortmund, Fakult\"at f\"ur
  Mathematik, Lehrstuhl X, Vogelpothsweg 87, 44227 Dortmund, Germany}
\email{paul.manns@tu-dortmund.de}

\author{Christian Meyer} \address{Christian Meyer, Technische Universit\"at Dortmund, Fakult\"at f\"ur
  Mathematik, Lehrstuhl X, Vogelpothsweg 87, 44227 Dortmund, Germany}
\email{christian2.meyer@tu-dortmund.de}

\date{\today} 

\begin{abstract} 
    This paper is concerned with an optimization 
    problem that is constrained by the Kantorovich
    optimal transportation problem. This bilevel
    optimization problem can be reformulated as a 
    mathematical problem with 
    complementarity constraints in the space of regular
    Borel measures.
    Because of the non-smoothness induced by the 
    complementarity relations, 
    problems of this type are frequently regularized.
    Here we apply a quadratic regularization of the 
    Kantorovich problem.
    As the title indicates,
    this is the second part 
    in a series of three papers.
    While the existence of optimal solutions to both
    the bilevel Kantorovich problem and its regularized 
    counterpart were shown in the first part, 
    this paper deals with the (weak-$\ast$) convergence
    of solutions to the regularized bilevel problem to 
    solutions of the original bilevel Kantorovich problem.
\end{abstract}

\maketitle


\section{Introduction}\label{sec:intro}

In this paper, we consider a bilevel optimization problem with the Kantorovich problem of 
optimal transport as its lower-level problem.
Given two non-negative marginals $\mu_1 \in \Measure(\Omega_1)$ and $\mu_2 \in \Measure(\Omega_2)$
on compact domains $\Omega_1$ and $\Omega_2$
with the same total mass and a continuous cost function
$c\in C(\Omega_1\times \Omega_2)$, 
the Kantorovich problem reads as follows:
\begin{equation}\tag{KP}\label{eq:KP}
    \left\{\quad
    \begin{aligned}
        \inf_{\pi} \quad & \KK(\pi) \coloneqq \int_\Omega c(x)\, \d \pi(x)\\
        \text{s.t.} \quad & \pi \in \Pi(\mu_1,\mu_2), \quad \pi \geq 0.
    \end{aligned}
    \quad \right.
\end{equation}
Herein, $\Pi(\mu_1, \mu_2)$ 
denotes the set of feasible couplings of $\mu_1$ and $\mu_2$ defined by
\begin{equation}
    \Pi(\mu_1, \mu_2) := \{\pi \in \Measure(\Omega_1 \times \Omega_2) :  
    {P_i}_\# \pi = \mu_i, \; i=1,2\},
\end{equation}
where $P_i$ denotes the projection on the $i$-th variable and ${P_i}_\# \pi$ is the push forward of 
$P_i$ w.r.t.\ $\pi$. 
The Kantorovich problem is a classical, well-established model for optimizing transportation processes and
we exemplarily refer to \cite{Vil03, Vil09, AG13, San15} for more details on its mathematical background.

For the definition of the bilevel problem, 
the cost function is set to $c = \cost \in C(\Omega_1\times \Omega_2)$ and we fix the second marginal 
$\mu_2 = \marg$, i.e., both $\cost$ and $\marg$
are given data.
The upper-level optimization variables are $\pi$ and $\mu_1$ 
such that the overall bilevel problem is given by
\begin{equation}\tag{BK}\label{eq:BKfoo}
    \left\{\quad
    \begin{aligned}
        \inf_{\pi,\mu_1} \quad & \JJ(\pi, \mu_1) \\
        \text{s.t.} \quad & \mu_1\in\Measure(\Omega_1), \quad \pi \in \Measure(\Omega_1 \times \Omega_2), \\
        & \mu_1\geq 0,  \quad \|\mu_1\|_{\Measure(\Omega_1)} = \|\marg\|_{\Measure(\Omega_2)}, \\
        & \text{$\pi$ solves \eqref{eq:KP} with $\mu_2 = \marg$ and $c = \cost$}.
    \end{aligned}
    \right.        
\end{equation}
A potential application of such a bilevel problem could,
for instance, be the identification of the marginal $\mu_1$ 
based on measurements of the transportation process.
The upper-level objective would then read
$\JJ(\pi, \mu_1) = \JJ(\pi):= \|\pi - \pi^\textup{d}\|_{\Measure(D)}$, 
where $D\subset \Omega_1 \times \Omega_2$ is a given observation domain and $\pi^\textup{d} \in \Measure(D)$ 
denotes the measured transport plan in $D$. Another example in form of an optimal control problem 
that fits into the framework of \eqref{eq:BKfoo} will be given at the end of the paper.

For its numerical solution, the Kantorovich problem is frequently regularized in order to avoid 
the ``curse of dimensionality'' caused by the discretization of the transport plan 
on the product space $\Omega_1 \times \Omega_2$. A prominent example is the entropic regularization
(see e.g.\ \cite{CLM21} for a convergence analysis in function space), 
which leads to the well-known Sinkhorn algorithm, cf.\ \cite{Cut13, CP18}.
An alternative regularization approach was proposed in \cite{LMM21}, where the squared 
$L^2(\Omega_1\times\Omega_2)$-norm of $\pi$ is added to the objective of \eqref{eq:KP}, weighted with 
a regularization parameter $\gamma > 0$. The advantageous implications of this approach are similar 
to the ones caused by the entropic regularization. First, the regularized counterpart of \eqref{eq:KP} 
is a strictly convex problem so that its solution is unique. Moreover, the regularization leads to 
a substantial reduction of the dimension since the dual 
problem is equivalent to a (non-smooth) system of 
equations in $L^2(\Omega_1) \times L^2(\Omega_2)$ instead of $L^2(\Omega_1\times \Omega_2)$. 
In \cite{LMM21}, a semi-smooth Newton-type method is employed to solve this system and 
the convergence for $\gamma \searrow 0$ (in a more general framework covering the quadratic regularization) 
is investigated in \cite{ML21}.
A further advantage (in comparison to the entropic regularization) of the quadratic regularization is that
it better preserves the sparsity pattern of the optimal transport plan as the numerical experiments in 
\cite{LMM21} demonstrate.
 
In view of the success of regularization techniques for the 
(numerical) solution of the Kantorovich problem, 
it is reasonable to apply them in the bilevel context too.
As the title indicates, we follow the quadratic approach from 
\cite{LMM21} and replace the Kantorovich problem in \eqref{eq:BKfoo} by its quadratically regularized
counterpart. In this context, the following questions 
naturally arise:
\begin{itemize}
    \item Does the bilevel problem \eqref{eq:BK} and its regularized counterpart admit (globally optimal)
    solutions?
    \item Do solutions of the regularized bilevel problems (or subsequences thereof) converge to solutions 
    of \eqref{eq:BK} for vanishing
    regularization parameter $\gamma \searrow 0$?
    \item How can we efficiently solve (discretized versions 
    of) the regularized bilevel problems?
\end{itemize}
While the first question is addressed in the predecessor paper \cite{kant1}, 
the present paper investigates the convergence behavior of such optimal solutions for $\gamma\searrow 0$.
So far, we are only able to show that (subsequences of) optimal solutions converge (weakly-$\ast$) 
to optimal solutions of the original problem \eqref{eq:BK} under fairly restrictive assumptions on 
the data $\marg$ and the structure of the objective. Nevertheless, at the end of the paper, we will see 
that there are relevant examples, where these assumptions are fulfilled.
At least in finite dimensions, these assumptions can be weakened, as we show in the third part of 
this series of papers, see~\cite{kant3}.
The third question, concerning an efficient and robust 
numerical solution of the regularized bilevel problems,
is subject to future research. Albeit regularized, the bilevel problems are still non-smooth, as the
necessary and sufficient optimality conditions associated 
with the regularized counterpart to \eqref{eq:KP}
involve the $\max$-operator, see~\cite[Theorem~2.11]{LMM21}. 
At the same time, this operator promotes the desired sparsity 
of the solution and for this reason, 
a further smoothing of the $\max$-operator should be avoided. 
We expect that algorithms, which are tailored to bilevel problems with non-smooth lower-level problem, behave
well in this setting, see, for instance, the approaches
in \cite{HS16, HU19, CdlRM20}.

The remainder of this paper is organized as follows: 
After introducing some basic notation and assumptions in the rest of this introduction, 
we collect some known results on the Kantorovich problem, its quadratic regularization
as well as the existence results from the companion paper \cite{kant1} in Section~\ref{sec:prelim}. 
The remaining part of the paper is then devoted to the convergence analysis for vanishing regularization
parameter $\gamma \searrow 0$.
First, in Section~\ref{sec:feas}, we show that weak-$\ast$ accumulation points of solutions of 
the regularized bilevel problems are feasible for \eqref{eq:BK}, i.e., in particular, 
the limit of the sequence of transport plans solves the Kantorovich problem associated with the limit of 
the marginals. Afterwards, in Section~\ref{sec:brenier}, we establish the optimality of the weak-$\ast$ limit 
under additional assumptions.
The paper ends with two application-driven examples,
where the additional assumptions are fulfilled.

\subsection{Notation and Standing Assumptions}\label{sec:assu}

Throughout the paper, the Euclidean norm of a vector
$a\in \R^n$, $n\in \N$, is denoted by $|a|$.
Moreover, the open ball in $\R^n$ of radius $r>0$ centered in $a$ is denoted by $B(a,r)$.
	
\subsubsection*{Domains}
For $d_1,d_2\in\N$, let $\Omega_1\subset\R^{d_1}$ and $\Omega_2\subset\R^{d_2}$ 
be compact sets with non-empty interior.
 We moreover suppose that 
their Cartesian product $\Omega\coloneqq\Omega_1\times\Omega_2$ 
coincides with the closure of its interior and has a Lipschitz boundary in the sense of \cite[Def.~1.2.2.1]{Gri85}.
The dimension of $\Omega$ is $d := d_1 + d_2$.
By $\Borel(\Omega)$ we denote the Borel $\sigma$-algebra on $\Omega$ 
and $\lambda$ is the Lebesgue measure on $\Borel(\Omega)$.
For $\Omega_1$ and $\Omega_2$, $\Borel(\Omega_i)$ and $\lambda_i$, $i=1,2$, are defined analogously 
so that $\lambda = \lambda_1 \otimes \lambda_2$.
Furthermore, we abbreviate $|\Omega_1|\coloneqq\lambda_1(\Omega_1)$, 
$|\Omega_2|\coloneqq\lambda_2(\Omega_2)$, and $|\Omega|\coloneqq\lambda(\Omega)$.

\subsubsection*{Marginals}
Let $(X, \Borel(X))$ be a measurable space. We denote by $\Measure(X)$ the space of 
(signed) regular Borel measures on $X$ equipped with the total variation norm, i.e., 
$\|\mu\|_{\Measure(X)} \coloneqq |\mu|(X)$.
By $P_i\colon\Omega_1 \times \Omega_2\ni(x_1,x_2)\mapsto x_i\in\Omega_i$, $i=1,2$, 
we denote the projection on the $i$-th variable. For the ease of notation,
we will denote projections with different domains and ranges by the same symbol, i.e.,
e.g., $P_2 : \Omega_1 \times \Omega_1 \ni (x_1, y_1) \mapsto y_1 \in \Omega_1$.
The respective domains and ranges will become clear from the context. 
If $\mu_1\in\Measure(\Omega_1)$ and $\mu_2\in\Measure(\Omega_2)$,
the set of transport plans between the marginals $\mu_1$ and $\mu_2$ is given by
\begin{equation}\label{eq:Pi}
    \Pi(\mu_1,\mu_2) \coloneqq \{\pi\in\Measure(\Omega)\colon {P_1}_\#\pi=\mu_1~\text{and}~{P_2}_\#\pi=\mu_2\},
\end{equation}
with the pushforward measure of $\pi$ via the projection $P_i$, $i=1,2$, being defined by
\begin{equation*}
    {P_i}_\#\pi\coloneqq\pi\circ P_i^{-1}\colon\Borel(\Omega_i)\to\R.
\end{equation*}		
Elements from $\Pi(\mu_1, \mu_2)$ are frequently called \emph{couplings} of $\mu_1$ and $\mu_2$.
Note that $\Pi(\mu_1,\mu_2)=\emptyset$ if $\mu_1(\Omega_1)\neq\mu_2(\Omega_2)$.
Throughout the paper, $\marg \in \Measure(\Omega_2)$ is a fixed marginal satisfying 
$\marg \geq 0$ and, in order to ease notation, $|\marg|(\Omega_2)  = 1$. 
The normalization condition is no restriction and can be ensured by re-scaling.
In order to shorten the notation, we write 
$\Prob(\Omega_i) := \{\mu \in \Measure(\Omega_i) : \mu \geq 0,\; |\mu|(\Omega_i) = 1\}$, $i=1,2$,
for the set of probability measures on $\Omega_i$.

Since $\Omega_1$, $\Omega_2$, and $\Omega$ are compact, 
the pre-dual spaces of $\Measure(\Omega_1)$, $\Measure(\Omega_2)$,  and $\Measure(\Omega)$
are $C(\Omega_1)$, $C(\Omega_2)$,  and $C(\Omega)$, respectively. 
We denote the associated dual pairings by $\dual{\mu}{v}$ and it will become clear from the context, which 
domain this refers to.

Given a measure space $(X, \AA, \mu)$, the Lebesgue space of $p$-times integrable functions is denoted 
by $L^p(X, \mu)$, $p\in [1, \infty)$. If $X\subset \R^n$, $n\in \N$, is a Lebesgue measurable set,
$\mathcal{A} = \Borel(X)$, and 
$\mu$ is the Lebesgue measure, we write $L^p(X)$.

\subsubsection*{Cost Function}
The cost function is assumed to satisfy $\cost \in W^{1,p}(\Omega)$, $p > d $, 
where, with a slight abuse of notation, $W^{1,p}(\Omega)$ denotes the Sobolev space on $\interior(\Omega)$. 
Note that, due to the regularity of $\partial\Omega$, $W^{1,p}(\Omega)$ is compactly embedded in 
$C(\Omega)$, cf.\ e.g.\ \cite[Theorem 6.3]{AF03}. 
Thus, there exists a continuous representative of $\cost$, which we denote by the same symbol.

\subsubsection*{Bilevel Objective}
The functional $\JJ \colon \Measure(\Omega) \times \Measure(\Omega_1)\to\R$ 
is supposed to be lower semicontinuous w.r.t.\ weak-$\ast$ convergence and bounded on bounded sets, i.e., 
for every $M > 0$, it holds that 
\begin{equation}
    \sup \{ |\JJ(\pi, \mu_1)|\colon\|(\pi, \mu_1)\|_{\Measure(\Omega) \times \Measure(\Omega_1)} \leq M \} < \infty.
\end{equation}
At the very end of the paper, several examples for objectives fulfilling these assumptions will be given.

\section{Preliminaries and Known Results}\label{sec:prelim}
In comparison to the bilevel problem discussed in \cite{kant1}, we consider a slightly different problem 
involving an additional constraint on the distance of $\supp(\mu_1)$ to the boundary $\partial\Omega_1$. 
This constraint is needed to ensure the weak-$\ast$ convergence of the mollified marginal, but can be avoided by 
passing on to an equivalent problem on an enlarged domain, see Lemma~\ref{lem:BKDelta} below. 
Our bilevel problem including the additional constraint
with distance parameter $\rho > 0$ reads as follows:
\begin{equation}\tag{BK}\label{eq:BK}
    \left\{\quad
    \begin{aligned}
        \inf_{\pi,\mu_1} \quad & \JJ(\pi, \mu_1) \\
        \text{s.t.} \quad & \mu_1\in \Prob(\Omega_1), \quad \dist(\supp(\mu_1), \partial\Omega_1) \geq \rho, \\
        & \pi \in \argmin\left\{\int_\Omega \cost\,\d\varphi
        \colon\varphi\in\Pi(\mu_1,\marg), \,\varphi\geq 0\right\}.
    \end{aligned}
    \right.        
\end{equation}
In the following, we tacitly assume that
$\rho > 0$ such that the feasible set of
\eqref{eq:BK} is non-empty.
Let us again mention that the lower-level problem only admits solutions provided that $\mu_1$ is a probability measure like $\marg$. 
To show the existence of a solution to this bilevel problem, we need the following result:

\begin{lemma}\label{lem:MMclosed}
    The set 
    \begin{equation}\label{eq:defsetM}
        \MM := \{ \mu_1 \in \Measure(\Omega_1) : \mu_1\geq 0, \; \dist(\supp(\mu_1), \partial\Omega_1) \geq \rho \}
    \end{equation}
    is closed w.r.t.\ weak-$\ast$ convergence. 
\end{lemma}

\begin{proof}
    Let a sequence $\{\mu_1^n\}_{n\in\N} \subset \MM$ 
    be given such that $\mu_1^n \weakly^* \mu_1$ in $\Measure(\Omega_1)$. It is clear that 
    weak-$\ast$ convergence gives $\mu_1 \geq 0$. For the remaining claim, 
    we argue by contradiction and assume that there is an $x\in \supp(\mu_1)$ with
    \begin{equation*}
        x \notin A := \{ \xi \in \Omega_1 : \dist(\xi, \partial\Omega_1) \geq \rho\}.
    \end{equation*}
    Due to the Lipschitz continuity of the distance function, $A$ is closed and thus there is an $r > 0$ such that 
    $B(x, r) \cap A = \emptyset$. By Urysohn's lemma, there is thus a continuous function 
    $\phi \in C(\Omega_1; [0,1])$ such that $\phi \equiv 0$ on $A$ and $\phi \equiv 1$ on 
    $\overline{B(x, r/2)} \cap \Omega_1$. Since $x \in \supp(\mu_1)$, this yields the desired contradiction:
    \begin{equation*}
        0 < \mu_1(B(x, r/2)) 
        \leq \int_{\Omega_1} \phi(\xi)\, \d \mu_1(\xi) 
        = \lim_{n\to \infty} \int_{\Omega_1} \phi(\xi)\, \d \mu_1^n(\xi) = 0,
    \end{equation*}     
    where the last equality follows from $\phi \equiv 0$ on $A \supset \supp(\mu_1^n)$.
\end{proof}

\begin{proposition}\label{prop:existBK}
    Under our standing assumptions, the bilevel Kantorovich problem \eqref{eq:BK} admits at least one globally optimal solution.
\end{proposition}

\begin{proof}
    In \cite[Theorem~3.2]{kant1}, the assertion is shown by means of the stability of transport plans according to 
    \cite[Theorem~5.20]{Vil09}
    for an analogous bilevel problem without the additional constraint $\dist(\supp(\mu_1), \partial\Omega_1) \geq \rho$. 
    Since the set $\MM$ from \eqref{eq:defsetM} is weakly-$\ast$ closed as seen in Lemma~\ref{lem:MMclosed}, 
    the proof of \cite[Theorem~3.2]{kant1} readily carries over to \eqref{eq:BK}.
\end{proof}

Since we do not only need to mollify $\mu_1$ but also $\marg$, we have to require an additional assumption 
on $\marg$ similar to the constraint on $\mu_1$ in \eqref{eq:BK}.

\begin{assumption}\label{assu:marg}
    There is a constant $\rho > 0$ such that $\dist(\supp(\marg), \partial\Omega_2) \geq \rho$.
\end{assumption}

As indicated above, due to the following elementary result, Assumption~\ref{assu:marg} 
as well as the additional constraint in \eqref{eq:BK} can be avoided.

\begin{lemma}\label{lem:BKDelta}
    Consider a bilevel Kantorovich problem of the form 
    \begin{equation}\label{eq:BK2}
    \left\{\quad
    \begin{aligned}
        \inf_{\pi,\mu_1} \quad & \JJ(\pi, \mu_1) \\
        \text{\textup{s.t.}} \quad & \mu_1\in \Prob(\Omega_1),  \quad
        \pi \in \argmin\left\{\int_\Omega \cost\,\d\varphi
        \colon\varphi\in\Pi(\mu_1,\marg), \,\varphi\geq 0\right\},
    \end{aligned}
    \right.        
    \end{equation}
    where $\marg \in \Prob(\Omega_2)$ need not necessarily satisfy Assumption~\ref{assu:marg}. If we define 
    $\Omega_i^\rho := \Omega_i + \overline{B(0,\rho)}$, $i = 1,2$, and 
    $\Omega^\rho := \Omega_1^\rho \times \Omega_2^\rho$, then \eqref{eq:BK2} is equivalent to 
    \begin{equation}\label{eq:BKDelta}
        \left\{\quad
        \begin{aligned}
            \inf_{\pi,\mu_1} \quad & \JJ(\pi|_\Omega, \mu_1|_{\Omega_1}) \\
            \text{\textup{s.t.}} \quad & \mu_1\in \Prob(\Omega_1^\rho), \quad 
            \dist(\supp(\mu_1), \partial\Omega_1^\rho) \geq \rho, \\
            & \pi \in \argmin\left\{\int_\Omega \cost\,\d\varphi
            \colon\varphi\in\Pi_\rho(\mu_1,\mu_2^\rho), \,\varphi\geq 0\right\} ,
        \end{aligned}
        \right.        
    \end{equation}
    where $\pi|_\Omega \in \Measure(\Omega)$ and $\mu_1|_{\Omega_1}\in \Measure(\Omega_1)$ 
    denote the restrictions of $\pi$ and $\mu_1$, while
    $\mu_2^\rho \in \Prob(\Omega_2^\rho)$ is the extension of $\marg$ to $\Omega_2^\rho$ by 
    zero, i.e., $\mu_2^\rho(B) := \marg(B \cap \Omega_2)$ for all $B\in \Borel(\Omega_2^\rho)$. 
    Moreover, we set
    \begin{equation*}
    \begin{aligned}
        \Pi_\rho(\mu_1,\mu_2^\rho) 
        \coloneqq \{\pi\in\Measure(\Omega^\rho)\colon 
         & \pi(A\times \Omega_2^\rho) = \mu_1(A) \;\; \forall\, A \in \Borel(\Omega_1^\rho),\\           
         & \pi(\Omega_1^\rho \times B) = \mu_2^\rho(B) \;\; \forall\, B \in \Borel(\Omega_2^\rho)\}.
    \end{aligned}
    \end{equation*}
    The above equivalence means that, if $\mu_1$ and $\pi$ 
    solve \eqref{eq:BK2}, their extensions by zero, 
    denoted by $\mu_1^\rho$ and $\pi^\rho$, also solve \eqref{eq:BKDelta}, whereas, 
    if $\mu_1$ and $\pi$ solve \eqref{eq:BKDelta}, their restrictions $\mu_1|_{\Omega_1}$
    and $\pi|_\Omega$ are solutions of \eqref{eq:BK2}, each time with the same optimal value.
\end{lemma}
\begin{proof}
    If $\mu_1$ and $\pi$ are feasible for \eqref{eq:BK2}, then 
    $\mu_1^\rho$ and $\pi^\rho$ satisfy $\dist(\supp(\mu_1^\rho), \partial\Omega_1^\rho) \geq \rho$
    and $\pi^\rho \in \Pi_\rho(\mu_1^\rho,\mu_2^\rho)$ by construction. Vice versa, if $\mu_1$ and $\pi$ are feasible for 
    \eqref{eq:BKDelta}, then $\mu_1|_{\Omega_1}$ and $\pi|_\Omega$ satisfy 
    $\pi|_\Omega \in \Pi(\mu_1|_{\Omega_1}, \marg)$, since $\pi \in\Pi_\rho(\mu_1,\mu_2^\rho)$ implies
    $\supp(\pi) \subset \supp(\mu_1) \times \supp(\mu_2^\rho) = \supp(\mu_1|_{\Omega_1}) \times \supp(\marg)$. 
    Therefore, because the objectives of the 
    lower-level problems in \eqref{eq:BK2} and 
    \eqref{eq:BKDelta} are the same, 
    the same holds for the feasible sets of
    \eqref{eq:BK2} and \eqref{eq:BKDelta}
    (after extension by zero and restriction,
    respectively). Since the upper-level objectives are 
    also identical, this gives the assertion. 
\end{proof}

Lemma~\ref{lem:BKDelta} shows that one can equivalently solve \eqref{eq:BKDelta} instead of \eqref{eq:BK2} 
and the former problem satisfies an assumption analogous to Assumption~\ref{assu:marg} and guarantees the additional 
constraint on the distance of $\supp(\mu_1)$ to the boundary. 
Without loss of generality, we can therefore assume Assumption~\ref{assu:marg} to hold and consider the modified problem 
\eqref{eq:BK} involving the additional constraint on $\supp(\mu_1)$. Consequently, we will tacitly take
Assumption~\ref{assu:marg} for granted in the rest of the paper.

\subsection{Quadratic Regularization}
We now turn to the quadratic regularization of \eqref{eq:BK}. Let us first introduce the regularized lower-level problem.
Given a regularization parameter $\gamma > 0$, 
two marginals $\mu_1 \in L^2(\Omega_1)$, $\mu_2 \in L^2(\Omega_2)$, and a cost function $c\in L^2(\Omega)$, 
we consider the following regularized counterpart to \eqref{eq:KP}:
\begin{equation}\tag{$\text{KP}_\gamma$}\label{eq:KPgam}
    \left\{\quad
    \begin{aligned}
        \inf_{\pi_\gamma} \quad & \KK_\gamma(\pi_\gamma) \coloneqq \int_\Omega c(x)\, \pi_\gamma(x)\, \d \lambda(x) 
        + \tfrac{\gamma}{2} \|\pi_\gamma\|_{L^2(\Omega)}^2 \\
        \text{s.t.} \quad & \pi_\gamma \in L^2(\Omega), \quad \pi_\gamma \geq 0 
        \quad\lambda\text{-a.e. in}~\Omega, \\
        & \int_{\Omega_2}\pi_\gamma(x_1, x_2) \d \lambda_2(x_2) 
        =\mu_1(x_1)\quad\lambda_1\text{-a.e.\ in}~\Omega_1, \\
		& \int_{\Omega_1}\pi_\gamma(x_1, x_2) \d \lambda_1(x_1) 
		= \mu_2(x_2)\quad\lambda_2\text{-a.e.\ in}~\Omega_2.
    \end{aligned}
    \right.
\end{equation}

\begin{lemma}[{\cite[Lemma~2.1]{LMM21}}]\label{lem:KPgamexist}
     Problem \eqref{eq:KPgam} admits a solution if and only if $\mu_i \geq 0$ $\lambda_i$-a.e.\ in $\Omega_i$, $i=1,2$, and 
    $\|\mu_1\|_{L^1(\Omega_1)} = \|\mu_2\|_{L^1(\Omega_2)}$. If a solution exists, then it is unique. 
\end{lemma}

Thanks to the above lemma, we can define the solution operator to \eqref{eq:KPgam}:
\begin{align}
    & \SS_\gamma : L^2(\Omega) \times\MM_0(\Omega) \ni (c, \mu_1, \mu_2)
    \mapsto \pi_\gamma \in L^2(\Omega), \label{eq:defSgam}\\
    & \MM_0(\Omega) 
    \coloneqq \big\{ (\mu_1, \mu_2) \in L^2(\Omega_1) \times L^2(\Omega_2) \colon 
    \begin{aligned}[t]
        & \|\mu_1\|_{L^1(\Omega_1)} = \|\mu_2\|_{L^1(\Omega_2)}, \\[-0.5ex]
        & \; \mu_i \geq 0\text{ $\lambda_i$-a.e.\ in } \Omega_i, i=1,2 \big\}.
    \end{aligned}        
\end{align}
What is more, if there exist constants $\cbound > -\infty$ and $\delta > 0$ such that 
$c \geq \cbound$ $\lambda$-a.e.\ in $\Omega$ and $\mu_i \geq \delta$ $\lambda_i$-a.e.\ in $\Omega_i$, $i=1,2$, then 
the dual problem to \eqref{eq:KPgam} admits a solution, too, see \cite[Theorem~2.11]{LMM21} and Lemma~\ref{lem:quadreg} below. 
Similarly to the original Kantorovich problem in \eqref{eq:KP}, this dual problem leads to a significant reduction of 
the dimension, since it is an optimization problem in $L^2(\Omega_1) \times L^2(\Omega_2)$ rather than in $L^2(\Omega_1 \times \Omega_2)$.

In order to ensure the existence of solutions to \eqref{eq:KPgam} as well as the associated dual variables for two
given marginals $\mu_i\in \Measure(\Omega_i)$, $i = 1, 2$, we introduce the convolution and constant shifting operators
\begin{equation}\label{eq:defTT}
    \TT_i^\delta \colon \Measure(\Omega_i) \ni \mu_i \mapsto 
    \varphi_i^\delta \ast \mu_i + \frac{\delta}{|\Omega_i|} \in L^2(\Omega_i),
    \quad i = 1,2.
\end{equation}
Herein, $\delta > 0$ is a smoothing parameter, $\varphi_i^\delta \in C_c^\infty(\R^{d_i})$ denotes the (symmetric) standard mollifier with 
$\|\varphi_i^\delta\|_{L^1(\R^{d_i})} = 1$ and support in $\overline{B_i(0, \delta)} \subset \R^{d_i}$, $i=1,2$. 
As a corollary of the classical convergence result for convolution of measures on the whole space, see e.g.\ 
\cite[Theorem~4.2.2]{ABM06}, we obtain the following result. It illustrates the utility of Assumption~\ref{assu:marg} 
and the additional constraint on $\mu_1$ in \eqref{eq:BK}.

\begin{lemma}\label{lem:convolconv}
    Let $\Lambda \subset \R^d$, $d\in \N$, be compact and assume that sequences 
    $\{\mu_n\}_{n\in \N} \subset \Measure(\Lambda)$ and $\{\delta_n\}_{n\in\N} \subset \R^+$ are given 
    such that $\mu_n \weakly^* \mu$, $\delta_n \searrow 0$, and
    \begin{equation}\label{eq:distineq}
        \dist(\supp(\mu_n), \partial\Lambda) \geq \rho > 0 \quad \forall\, n\in \N.
    \end{equation}
    Then $\varphi^{\delta_n} \ast \mu_n \weak^\ast \mu$
    in $\Measure(\Lambda)$ as $n \to \infty$ .
\end{lemma}
\begin{proof}
    Let $v\in C(\Lambda)$ be arbitrary and denote $\Lambda_{\rho} := \{x\in \Lambda : \dist(x, \partial\Lambda) \geq \rho\}$. 
    Then Fubini's theorem along with \eqref{eq:distineq} yields
    \begin{equation*}
        \int_{\Lambda} v(x) (\varphi^{\delta_n} \ast \mu_n)(x)  \, \d x
        = \int_{\Lambda_\rho} (\varphi^{\delta_n} \ast v)(\xi) \, \d \mu_n(\xi) \to \int_{\Lambda} v(\xi) \d\mu(\xi),
    \end{equation*}
    where we used the uniform convergence of $\varphi^{\delta_n} \ast v$ on the 
    compact subset $\Lambda_\rho$ of $\interior(\Lambda)$.
\end{proof}

We note that the restrictions on
$\Omega$, non-empty interior and 
Lipschitz boundary, are not required for Lemma 
\ref{lem:convolconv} because $A_\rho \subset \interior(\Lambda)$ is
ensured by $\dist(\supp \mu_n,\partial \Lambda) \ge \rho$.

According to Lemma~\ref{lem:KPgamexist}, \eqref{eq:KPgam} only admits a solution if 
the total mass of the marginals is the same. In context of the bilevel problem \eqref{eq:BKgam} below, 
this is ensured for the smoothed marginals by Assumption~\ref{assu:marg} and the additional constraint on $\supp(\mu_1)$, provided that $\delta \leq \rho$.
Since these imply $\supp(\varphi_2^\delta \ast \marg) \subset \Omega_2$ and  
$\supp(\varphi_1^\delta \ast \mu_1) \subset \Omega_1$, respectively, we obtain for every $\mu_1 \in \Prob(\Omega_1)$ with 
$\dist(\supp(\mu_1), \partial\Omega_1) \geq \rho$ that
\begin{equation}\label{eq:equalmass}
\begin{aligned}
    \|\TT_1^\delta(\mu_1)\|_{L^1(\Omega_1)}
    &=  \int_{\Omega_1} \varphi_1^\delta \ast \mu_1 \,\d\lambda_1 + \delta \\[-1ex]
    &= \|\mu_1\|_{\Measure(\Omega_1)} + \delta
    = \|\marg\|_{\Measure(\Omega_2)} + \delta
    =  \| \TT_2^\delta(\marg) \|_{L^1(\Omega_2)}
\end{aligned}
\end{equation}
and consequently, \eqref{eq:KPgam} is well defined with the marginals $\TT_1^\delta(\mu_1)$ and 
$\TT_2^\delta(\marg)$.
We are now in the position to state the regularized version of \eqref{eq:BK}:
\begin{equation}\tag{BK$_{\gamma}^\delta$}\label{eq:BKgam}
    \left\{\quad
    \begin{aligned}
        \inf_{\pi_\gamma,\mu_1, c} \quad & \JJ_\gamma(\pi_\gamma, \mu_1, c) 
        \coloneqq \JJ(\pi_\gamma, \mu_1) + \tfrac{1}{p \gamma}\, \|c - \cost\|_{W^{1,p}(\Omega)}^p \\
        \text{s.t.} \quad & c\in W^{1,p}(\Omega), \quad
        \mu_1\in \Prob(\Omega_1), \quad \dist(\supp(\mu_1), \partial\Omega_1) \geq \rho, \\
        & \pi_\gamma 
        = \SS_\gamma\big( c, \TT_1^\delta(\mu_1), \TT_2^\delta(\marg) \big).
    \end{aligned}
    \right.        
\end{equation}
Here and in the following, with a slight abuse of notation, we denote the measure induced by the $L^2$-function $\pi_\gamma$ 
by means of the $L^2(\Omega)$-scalar product by the same symbol.
In comparison to \eqref{eq:BK}, we do not only replace the 
lower-level Kantorovich problem by its regularized 
counterpart, but also add the cost function $c$ to the set of optimization variables. 
This is motivated by the so-called reverse approximation property, which  
requires a set of optimization variables that is sufficiently rich as also observed,
e.g., in case of the optimization of perfect plasticity, 
see \cite{MW21}. For this reason, $c$ is treated as an 
additional optimization variable to have more flexibility 
at this point.
In the companion paper \cite{kant3}, this will be the essential tool to establish the 
reverse approximation property in finite dimensions.
The penalty term in the upper-level objective $\JJ_\gamma$ will ensure that, in the limit, $c$ equals the given 
costs $\cost$, see \eqref{eq:convcost} below.

\begin{proposition}\label{prop:BKgamexist}
    Let $\gamma > 0$ and $\delta \in (0,\rho]$ be fixed.
    There exists at least one globally optimal solution to the regularized bilevel Kantorovich problem \eqref{eq:BKgam}.
\end{proposition}

\begin{proof}
    The existence of solutions for a slightly different problem has been shown in \cite[Theorem~4.7]{kant1}, which differs from 
    \eqref{eq:BKgam} as follows:
    First, the bilevel problem in \cite{kant1} does not contain the additional constraint on $\supp(\mu_1)$, but,     
    similarly to the proof of Proposition~\ref{prop:existBK}, this constraint can easily be incorporated into the 
    existence proof using Lemma~\ref{lem:MMclosed}.
    Secondly, the bilevel problem is posed in $\Omega_1^\delta \times \Omega_2^\delta$ with 
    $\Omega_i^\delta := \Omega_i + \overline{B(0, \delta)}$, $i=1,2$. This ensures that the 
    marginals $\TT_1^\delta(\mu_1)$ and $\TT_2^\delta(\marg)$ have the same total mass. 
    In our case, however, this is guaranteed by Assumption~\ref{assu:marg} and the constraint on $\supp(\mu_1)$ together with $\delta \le \rho$,
    see~\eqref{eq:equalmass}. With the equality of the total mass of the marginals at hand,     
    the remaining part of the existence proof is then completely along the lines of \cite[Theorem~4.7]{kant1}.
\end{proof}
\begin{remark}\label{rem:BKexist_domain}
The restrictions on $\Omega$ other than compactness,
in particular the Lipschitz boundary, are only required
in order to ensure the (compact) embedding
$W^{1,p}(\Omega) \embed C(\Omega)$ that is
used in Proposition \ref{prop:existBK}
(\cite[Theorem~4.7]{kant1}).
If one can drop the term $\frac{1}{p\gamma}\|c - \cost\|_{W^{1,p}(\Omega)}^p$, all arguments up to this point
are valid under the assumption that
$\Omega_1$ and $\Omega_2$ are compact.
\end{remark}

\section{Convergence for Vanishing Regularization}

In the following, we investigate the behavior of optimal solutions of the regularized bilevel problem 
\eqref{eq:BKgam} for regularization parameters $\gamma$, $\delta$ tending to zero.
For this purpose, assume that we are given sequences of non-negative regularization and smoothing parameters
$\{\gamma_n\}_{n \in \N}$, $\{\delta_n\}_{n \in \N} \subset \R_+$ satisfying $\gamma_n, \delta_n \searrow 0$ as well as
\begin{equation}
	\label{eq:RegParamSeq}
	0<\delta_n\leq\rho\quad\text{for all}~n\in\N
	\quad\text{and}\quad
	\frac{\gamma_n}{\delta_n^{d}}
	\to 0
	~\text{as}~n\to\infty.
\end{equation}
The reason for the coupling of the
parameters $\delta_n$ and $\gamma_n$
will become clear in the proof of 
Proposition~\ref{prop:WeakLimFeasBilevel} below. To shorten the notation, we write (BK$_n$) instead of 
\hyperref[eq:BKgam]{(BK$_{\gamma_n}^{\delta_n}$)} for the regularized bilevel problem
associated with $\gamma_n$ and $\delta_n$.
Similarly, from now on, we equip all entities and variables that depend on either $\gamma_n$ or $\delta_n$ (or both) 
only with the index $n$, i.e., e.g.\ $\SS_n := \SS_{\gamma_n}$, 
$\TT_1^n := \TT_1^{\delta_n}$ and so on.
For each $n\in \N$, Proposition~\ref{prop:BKgamexist} ensures the existence of a solution 
$( \bar{\pi}_n, \bar{\mu}_1^n, \bar{c}_n )$ to \hyperref[eq:BKgam]{(BK$_n$)}.
Owing to the feasibility of $\bar{\mu}_1^n$, we find that $\|\bar{\mu}_1^n\|_{\Measure(\Omega_1)} = 1$. 
Moreover, the constraints in \eqref{eq:KPgam} imply
\begin{equation}
\begin{aligned}
\label{eq:PiBarEst}
	\|\bar{\pi}_n\|_{\Measure(\Omega)}
	&= \int_{\Omega_1} \int_{\Omega_2} \bar \pi_n \,\d\lambda_2 \,\d\lambda_1 \\
	&= \int_{\Omega_1} \TT_1^n(\bar\mu_1^n) \,\d\lambda_1
	= \|\varphi_1^n\|_{L^1(\R^{d_1})} \|\bar{\mu}_1^n\|_{\Measure(\Omega_1)} + \delta_n \leq 1 + \rho
\end{aligned}
\end{equation}
for all $n \in \N$, where we make use of $\bar \pi_n \geq 0$, \eqref{eq:equalmass},
and \eqref{eq:RegParamSeq}.
Hence, there is a subsequence (denoted by the same symbol to ease notation) such that
\begin{equation*}
	( \bar{\pi}_n, \bar{\mu}_1^n )
	\weakly^*
	( \bar{\pi}, \bar{\mu}_1 )
	\quad \text{in}~
	\Measure(\Omega) \times \Measure(\Omega_1)
	\quad \text{as } n \to \infty.
\end{equation*}
Furthermore, take an arbitrary, but fixed $\mu_1 \in \Prob(\Omega_1)$ and consider 
the regularized optimal transport plans 
$\pi_n = \SS_n ( \cost, \TT_1^n(\mu_1), \TT_2^n( \marg ) )$ for $n \in \N$.
Then, the triple $( \pi_n, \mu_1, \cost)$ is feasible for $(\hyperref[eq:BKgam]{\text{BK}_n})$ 
and $( \pi_n, \mu_1 )_{n \in \N}$ is bounded in $\Measure(\Omega) \times \Measure(\Omega_1)$, cf.\ \eqref{eq:PiBarEst}. 
The optimality of $( \bar{\pi}_n, \bar{\mu}_1^n, \bar{c}_n )$ for (\hyperref[eq:BKgam]{BK$_n$}) thus yields
\begin{equation*}
\begin{aligned}
	\|\bar c_n - \cost\|_{W^{1,p}(\Omega)}^p
	\leq p\,\gamma_n \big( \JJ( \pi_n, \mu_1 ) - \JJ( \bar{\pi}_n, \bar{\mu}_1^n ) \big)
	\leq \gamma_n\,C
\end{aligned}
\end{equation*}
for all $n \in \N$ with some constant $C > 0$, because $|\JJ|$ is bounded on bounded sets by assumption. Hence,
we obtain for the whole sequence $\{\bar c_n\}_{n\in \N}$ (and not just the subsequence) that
\begin{equation}\label{eq:convcost}
	\bar c_n \to \cost \quad \text{in}~W^{1,p}(\Omega) \text{ as } n \to \infty.
\end{equation}
Now that we have found an accumulation point $(\bar{\pi},\bar{\mu}_1,\cost)$ 
of the sequence of regularized solutions $\{( \bar{\pi}_n, \bar{\mu}_1^n, \bar{c}_n )\}_{n \in \N}$, 
we aim to show its optimality for the original bilevel Kantorovich problem \eqref{eq:BK}.
We start with the feasibility of $(\bar{\pi},\bar{\mu}_1)$ for \eqref{eq:BK} in the next subsection.


\subsection{Feasibility of the Limit Plan}\label{sec:feas}

\begin{lemma}\label{lem:WeakLimFeasKant}
    Let $\{( \pi_n, \mu_1^n, c_n )\}_{n \in \N} \subset \Measure(\Omega) \times \Measure(\Omega_1) \times W^{1,p}(\Omega)$ 
    be a sequence of feasible points for the regularized bilevel problems 
    \textnormal{(\hyperref[eq:BKgam]{BK$_n$})}, $n \in \N$. 
	If  $( \pi, \mu_1) \in \Measure(\Omega) \times \Measure(\Omega_1)$ is 
	a weak-$\ast$ accumulation point of $\{( \pi_n, \mu_1^n)\}$, then $\pi$ is a non-negative coupling between 
	$\mu_1$ and $\marg$, i.e., $\pi \in \Pi(\mu_1, \marg)$.
\end{lemma}

\begin{proof}
    In order to avoid double subscripts, we assume w.l.o.g.\ that the whole sequence converges.
    The non-negativity of $\pi_n$ carries over to the weak-$\ast$ limit $\pi$. 
    It remains to show that $\pi$ is a coupling of $\mu_1$ and $\marg$. 
    For this purpose, let $\phi_1 \in C(\Omega_1)$ be arbitrary but fixed. Then, the equality constraints in \eqref{eq:KPgam} imply
    \begin{equation*}
    \begin{aligned}
        \dual{\pi}{\phi_1 \circ P_1} 
        & = \lim_{n\to\infty} \dual{\pi_n}{\phi_1\circ P_1} \\
        & = \lim_{n\to\infty} 
        \int_{\Omega_1}  \phi_1(x_1) \int_{\Omega_2^n} \pi_n(x_1,x_2)\,\d\lambda_2(x_2)\,\d\lambda_1(x_1) \\
        & = \lim_{n\to\infty} \int_{\Omega_1}  \phi_1(x_1) \,\TT_1^n(\mu_1^n)(x_1) \,\d\lambda_1(x_1) \\
        & = \lim_{n\to\infty}
        \int_{\Omega_1} \phi_1(x_1) (\varphi_1^n \ast \mu_1^n)(x_1)  \,\d\lambda_1(x_1) + \frac{\delta_n}{|\Omega_1|} 
        \int_{\Omega_1} \phi_1(x_1)\,\d\lambda_1(x_1) \\
        &= \dual{\mu_1}{\phi_1},
    \end{aligned}
    \end{equation*}
    where we use Lemma~\ref{lem:convolconv} and the additional constraint on $\supp(\mu_1)$ in \eqref{eq:KPgam} 
    for the passage to the limit.
    Since $\phi_1$ was arbitrary, this implies ${P_1}_\# \pi = \mu_1$ as desired.
    An analogous argument for an arbitrary $\phi_2 \in C(\Omega_2)$ 
    shows ${P_2}_\#{\pi} = \marg$.
\end{proof}

\begin{lemma}\label{lem:TransPlanApprox}
    Let $\{\mu_1^n\}_{n\in\N} \subset \Prob(\Omega_1)$ be given such that $\mu_1^n \weakly^* \mu_1$.
    Moreover, let $\pi \in \Pi( \mu_1, \marg )$ be a non-negative coupling between $\mu_1$ and $\marg$. 
    Then, there exists a sequence of non-negative couplings $\pi_n \in\Pi(\mu_1^n,\marg)$ 
    that converges weakly-$\ast$ to $\pi$.
\end{lemma}

\begin{proof}
    The proof relies on the gluing lemma in combination with the equivalence of  weak-$\ast$ convergence
    and convergence in the Wasserstein-1-distance
    on compact domains. 
    We first note that, according to \cite[Theorem 4.1]{Vil09}, for each $n\in\N$,
    there exists an optimal coupling $\theta_n\in\Pi(\mu_1^n, \mu_1)$ 
    between $\mu_1^n$ and $\mu_1$ with respect to the metric cost $|x_1-y_1|$ on
    $\Omega_1 \times \Omega_1$. 
    Furthermore, thanks to the gluing lemma, see, e.g., \cite[Lemma 5.5]{San15}, there exist non-negative measures 
    $\sigma_n\in\Measure(\Omega_1\times\Omega_1\times\Omega_2)$ such that 
    ${P_{12}}_\#{\sigma_n}=\theta_n$ and ${P_{23}}_\#{\sigma_n}=\pi$ for all $n\in \N$.
    Here and in the following, 
    \begin{equation*}
        P_{jk}\colon \Omega_1\times\Omega_1\times\Omega_2 \to \Omega_1 \times \Omega_\ell, 
        \quad j, k=1,2,3, \; j< k, \; 
        \ell = k-1
    \end{equation*}    
    refers to the projection onto the coordinates $j$ and $k$.
	Using this projection, we define
	\begin{equation*}
		\pi_n \coloneqq {P_{13}}_\#{\sigma_n}\in\Measure(\Omega_1\times\Omega_2).
	\end{equation*}
	Then, by construction, we obtain for all $B_1\in\Borel(\Omega_1)$
	\begin{align*}
		({P_1}_\#{\pi_n})(B_1) &= \sigma_n\big(P_{13}^{-1}(B_1 \times \Omega_2)\big) \\
		&= \sigma_n(B_1\times\Omega_1\times\Omega_2) \\
		&= \sigma_n\big(P_{12}^{-1}(B_1\times \Omega_1)\big) = ({P_1}_\#{\theta_n})(B_1) = \mu_1^n(B_1)
	\end{align*}
	and analogously, for all $B_2\in\Borel(\Omega_2)$,
	\begin{equation*}
		({P_2}_\#{\pi_n})(B_2) = \sigma_n(\Omega_1\times\Omega_1\times B_2) 
		= ({P_2}_\#{\pi})(B_2) = \marg(B_2)
	\end{equation*}
	so that $\pi_n\in\Pi(\mu_1^n,\marg)$ as desired. 
	Moreover, the non-negativity of $\pi_n$ directly follows from the non-negativity of $\sigma_n$.
	
	To show the weak-$\ast$ convergence, we borrow an argument from the proof of \cite[Theorem 3.1]{BP21}.
	For this purpose, define the mapping
	\begin{equation*}
		P_{1323}\colon\Omega_1\times\Omega_1\times\Omega_2
        \ni (x_1,y_1,x_2) \mapsto \big((x_1,x_2),(y_1,x_2)\big) \in \Omega\times\Omega,
	\end{equation*}
	as well as $\zeta \coloneqq {P_{1323}}_\#{\sigma_n}$. We observe that $\zeta\in\Measure(\Omega\times\Omega)$ and
	\begin{equation*}
		({P_1}_\#{\zeta})(B) = \zeta(B\times\Omega) 
		= \sigma_n\big(P_{1323}^{-1}(B\times\Omega)\big) = \sigma_n\big(P_{13}^{-1}(B)\big) = \pi_n(B)
	\end{equation*}
	as well as
	\begin{align*}
		\big({P_2}_\#{\zeta}\big)(B) = \zeta(\Omega\times B) 
		= \sigma_n\big(P_{1323}^{-1}(\Omega\times B)\big) = \sigma_n\big(P_{23}^{-1}(B)\big) = \pi(B)
	\end{align*}
	for all $B\in\Borel(\Omega)$ so that $\zeta\in\Pi(\pi_n,\pi)$. 
	Again, the non-negativity of $\zeta$ directly follows from the non-negativity of $\sigma_n$. 
    Now we have everything at hand to estimate the Wasserstein-1-distance of $\pi_n$ and $\pi$:
	\begin{align*}
		0\leq W_1(\pi_n,\pi) 
		&= \inf_{0\leq\theta\in\Pi(\pi_n,\pi)} \int_{\Omega\times\Omega}|x-y|\, \d{\theta(x,y)} \\
		&\leq \int_{\Omega\times\Omega} |x-y| \,\d{\zeta(x,y)} \\
		&\leq \int_{\Omega\times\Omega} |(x_1,x_2) - (y_1, y_2)|\,\d ({P_{1323}}_\#{\sigma_n})((x_1,x_2),(y_1,y_2))\\
		&=  \int_{\Omega_1\times\Omega_1\times\Omega_2} |x_1-y_1|\,\d{\sigma_n(x_1,y_1,x_2)} \\
		&= \int_{\Omega_1\times\Omega_1}|x_1-y_1|\,\d{({P_{12}}_\#{\sigma_n})(x_1,y_1)} \\
		&= \int_{\Omega_1\times\Omega_1}|x_1-y_1|\,\d{\theta_n} = W_1(\mu_1^n,\mu_1) \to 0
		\quad\text{as } n \to \infty,
	\end{align*}
	where we use $\mu_1^n \weakly^* \mu_1$ by assumption and the equivalence of 
	weak-$\ast$ convergence and convergence in the Wasserstein-1-distance on compact domains 
	according to \cite[Theorem 6.9]{Vil09}. Using this equivalence once more
	finally yields $\pi_n \weakly^*\pi $ in $\Measure(\Omega)$ as $n \to \infty$.
\end{proof}

\begin{proposition}\label{prop:WeakLimFeasBilevel}
    Assume that the vanishing sequences $\{\gamma_n\}_{n\in \N}$ and $\{\delta_n\}_{n\in\N}$ satisfy \eqref{eq:RegParamSeq}.
	Let $\{( \pi_n, \mu_1^n, c_n )\}_{n \in \N} \subset \Measure(\Omega) \times \Measure(\Omega_1) \times W^{1,p}(\Omega)$ 
	be a sequence of feasible points for the regularized bilevel problems \textnormal{(\hyperref[eq:BKgam]{BK$_n$})}, $n \in \N$. 
	If  $( \pi, \mu_1, \cost ) \in \Measure(\Omega) \times \Measure(\Omega_1) \times W^{1,p}(\Omega)$ is an 
	accumulation point of this sequence w.r.t.\ weak-$\ast$ convergence in $\Measure(\Omega) \times \Measure(\Omega_1)$ 
	and weak convergence in $W^{1,p}(\Omega)$, 
	then $( \pi, \mu_1 )$ is feasible for \eqref{eq:BK}, i.e., $\mu_1 \in \Prob(\Omega_1)$, 
    $\dist(\supp(\mu_1), \partial\Omega_1) \geq \rho$,	
	 and $\pi$ is optimal for \eqref{eq:KP} with respect to the marginals $\mu_1$ and $\marg$ 
	 as well as the cost function $\cost$.
\end{proposition}

\begin{proof}
    In order to avoid double subscripts, we again assume w.l.o.g.\ that the whole sequence converges.
    Since $\Prob(\Omega_1)$ as well as the set $\MM$ from \eqref{eq:defsetM} are weakly-$\ast$ closed, 
    see Lemma~\ref{lem:MMclosed}, the properties
    for $\mu_1$ follow immediately. 

    As we have already seen in Lemma \ref{lem:WeakLimFeasKant}, 
    $\pi$ is feasible for the Kantorovich problem \eqref{eq:KP} with respect to $\mu_1$ and $\marg$. 
    So, it suffices to show the optimality of $\pi$ for \eqref{eq:KP}. To this end, 
    recall the lower-level problems from the feasible sets of \textnormal{(\hyperref[eq:BKgam]{BK$_n$})} 
    that are solved by $\pi_n$:
    \begin{equation}\tag{$\text{KP}_n$}\label{eq:KPn}
    \left\{\quad
    \begin{aligned}
        \min_{\pi} \quad & \KK_n(\pi) \coloneqq \int_{\Omega} c_n\, \pi \, \d \lambda
        + \tfrac{\gamma_n}{2} \|\pi\|_{L^2(\Omega)}^2 \\
        \text{s.t.} \quad & \pi \in L^2(\Omega), \quad \pi \geq 0 
        \quad\lambda\text{-a.e. in}~\Omega, \\
        & \int_{\Omega_2} \pi(x_1, x_2) \d \lambda_2(x_2) 
        =\TT_1^n(\mu_1^n)(x_1)\quad\lambda_1\text{-a.e.\ in}~\Omega_1, \\
		& \int_{\Omega_1} \pi(x_1, x_2) \d \lambda_1(x_1) 
		= \TT_2^n(\marg)(x_2)\quad\lambda_2\text{-a.e.\ in}~\Omega_2.
    \end{aligned}
    \right.
    \end{equation}
    By \cite[Theorem~4.1]{Vil09}, we know that there is at least one solution of the Kantorovich problem~\eqref{eq:KP} 
    associated with $\mu_1$, $\marg$, and the limit cost function $\cost$. We consider an arbitrary of these solutions  
	and denote it by $\pi^* \in \Measure(\Omega)$.
	Owing to Lemma~\ref{lem:TransPlanApprox}, there exists a sequence of non-negative couplings 
	$\{\pi_n^*\}_{n\in\N}$ between $\mu_1^n$ and $\marg$ that converges weakly-$\ast$ to $\pi^*$. 
    We then define
    \begin{equation*}
    \begin{aligned}
        & \varphi_n(x_1,x_2) 
		 :=
        \varphi_1^n(x_1)\, \varphi_2^n(x_2), 
        \quad (x_1,x_2) \in \Omega,\\
        \text{and} \;\;  &\vartheta_n^* 
        := \varphi_n\ast\pi_n^* + \frac{\delta_n}{|\Omega_1|\,|\Omega_2|} 
        = \int_\Omega \varphi_n(\xi- \cdot\,) \,\d\pi_n^*(\xi) + \frac{\delta_n}{|\Omega_1|\,|\Omega_2|} 
        \in L^2(\Omega).    
    \end{aligned}
    \end{equation*}
	Then, the non-negativity of $\pi_n^*$ implies the positivity of $\vartheta_n^*$. Moreover, the definition of $\vartheta^*_n$ 
	in combination with Fubini's theorem yields 
	\begin{equation*}
		\int_{\Omega_2} \vartheta_n^*\d{\lambda_2}
		= \varphi_1^n\ast\mu_1^n + \frac{\delta_n}{|\Omega_1|}
		= \TT_1^n(\mu_1^n),\quad 
		\int_{\Omega_1} \vartheta_n^*\d{\lambda_1}
		= \varphi_2^n\ast\marg + \frac{\delta_n}{|\Omega_2|}
		= \TT_2^n(\mu_2^n)
	\end{equation*}
	so that $\vartheta_n^*$ is feasible for \eqref{eq:KPn}. 
	For the objective of the Kantorovich problem, the optimality of $\pi_n$ for \eqref{eq:KPn} yields
	\begin{equation}\label{eq:cnconv1}
	\begin{aligned}
	    \dual{\cost}{\pi}
	    = \lim_{n\to\infty} \dual{c_n}{\pi_n} 
	    &\leq \liminf_{n\to\infty} \int_{\Omega} c_n \,\pi_n\, \d\lambda +
	    \tfrac{\gamma_n}{2} \,\|\pi_n\|_{L^2(\Omega)}^2 \\
	    &\leq \limsup_{n\to \infty}  \int_{\Omega} c_n \,\vartheta_n^*\, \d\lambda +
	    \tfrac{\gamma_n}{2} \,\|\vartheta_n^*\|_{L^2(\Omega)}^2.
	\end{aligned}
	\end{equation}
	Let us investigate the two addends on the right hand side of this inequality separately.
	Since $c_n \weakly \cost$ in $W^{1,p}(\Omega)$ and the embedding $W^{1,p}(\Omega) \embed C(\Omega)$
	is compact due to $p > d$, $c_n$ converges uniformly to $\cost$ in $\Omega$.
	Let us define $\Omega_\rho := \{x \in \Omega : \dist(x, \partial\Omega) \geq \rho\}$. 
	Then, the uniform convergence of both $c_n$ and the convolution in compact subsets of $\interior(\Omega)$ yields
	\begin{equation*}
    	\max_{x\in \Omega_\rho} 
    	\Big| \int_{\Omega} c_n(\xi) \,\varphi(x-\xi)\, \d\lambda(\xi) - \cost(x) \Big| = 0.
	\end{equation*}
	Since $\supp(\pi_n^*) \subset \supp(\mu_1^n) \times \supp(\marg) \subset \Omega_\rho$, this 
	in combination with the definition of $\vartheta^*_n$ and the weak-$\ast$ convergence of $\pi_n^*$ implies
	\begin{equation}\label{eq:cnconv2}
	\begin{aligned}
	    \int_{\Omega} c_n \,\vartheta_n^*\, \d\lambda 
	    &= \int_{\Omega_\rho} \int_{\Omega} c_n(x) \,\varphi(x-\xi) \d\lambda(x) \,\d\pi_n^*(\xi) \\[-1ex]
	    &\qquad\qquad\qquad\qquad 
	    + \frac{\delta_n}{|\Omega_1|\, |\Omega_2|} \int_\Omega c_n(x)\,\d\lambda(x) 
	    \to \dual{\cost}{\pi^*}.
	\end{aligned}
	\end{equation}
	For the second addend on the right hand side of \eqref{eq:cnconv1}, we obtain
	\begin{equation*}
       \|\vartheta_n^*\|_{L^2(\Omega)}^2
       \le 2 \int_{\Omega} \int_\Omega \varphi(x - \xi) \,\d\pi_n^*(\xi)^{2} \,\d\lambda(x) + 2 \delta_n^2 \\
       \leq 2\|\varphi_n\|_{L^2(B(0, \delta_n))}^2 \|\pi_n^*\|_{\Measure(\Omega)}^2 + 2\delta_n^2, 
	\end{equation*}
	where the $L^2$-norm of the standard mollifier is estimated by
	\begin{equation*}
	\begin{aligned}
	    \|\varphi_n\|_{L^2(B(0, \delta_n))}^2
	    &= \|\varphi_1^n\|_{L^2(B_1(0, \delta_n))}^2\,\|\varphi_2^n\|_{L^2(B_2(0, \delta_n))}^2 \\
	    &\leq \prod_{i=1}^2 \|\varphi_i^n\|_{L^\infty(B_i(0, \delta_n))} \, \|\varphi_i^n\|_{L^1(B_i(0, \delta_n))}
	    \leq C\,\delta_n^{-d_1 - d_2} = C\,\delta_n^{-d}
	\end{aligned}
	\end{equation*}
	with a constant $C>0$. In view of the coupling of $\gamma_n$ and $\delta_n$ in \eqref{eq:RegParamSeq} 
	and $\|\pi_n^*\|_{\Measure(\Omega)} = 1$ for all $n\in\N$, we thus arrive at 
    \begin{equation*}
        \frac{\gamma_n}{2} \,\|\vartheta_n^*\|_{L^2(\Omega)}^2
        \leq C \Big(\frac{\gamma_n}{\delta_n^{d}} + \gamma_n\,\delta_n^2\Big) \to 0.
    \end{equation*}
    Inserting this together with \eqref{eq:cnconv2} in \eqref{eq:cnconv1} implies 
    $\dual{\cost}{\pi} \leq \dual{\cost}{\pi^*}$ and, since $\pi$ is feasible for \eqref{eq:KP} associated with 
    $\mu_1$, $\marg$, and $\cost$, as seen above, while $\pi^*$ is optimal for that problem, 
    $\pi$ is optimal, too, and thus $(\pi, \mu_1)$ is feasible for \eqref{eq:BK} as claimed.
\end{proof}

Recall the sequence of solutions
$(\bar\pi_n, \bar\mu_1^n, \bar c_n)$  to the regularized bilevel problems from the beginning of this section. 
We already know that this sequence admits a weak-$\ast$ accumulation point. To be more precise 
$(\bar\pi_n, \bar\mu_1^n) \weakly^* (\bar\pi, \bar\mu_1)$ in $\Measure(\Omega) \times \Measure(\Omega_1)$
after possibly restricting to a subsequence, while 
the whole sequence of cost functions $\bar c_n$ converges strongly in $W^{1,p}(\Omega)$ to $\cost$. 
Thus, according to Proposition~\ref{prop:WeakLimFeasBilevel}, the weak-$\ast$ limit $(\bar\pi, \bar\mu_1)$ 
is feasible for \eqref{eq:BK}. As this observation holds for every accumulation point, 
we immediately obtain the following
\begin{corollary}\label{cor:ClusterPointOptimal}
    Assume that the vanishing 
    sequences $\{\gamma_n\}_{n\in \N}$ and $\{\delta_n\}_{n\in\N}$ satisfy \eqref{eq:RegParamSeq}.
	Suppose moreover that $(\pi^*,\mu_1^*)$ is a solution to the bilevel problem \eqref{eq:BK} that admits a \emph{recovery sequence}
    in the following sense: There is a sequence 
    $(\pi_n^*,\mu_{1,n}^*, c_n^*)_{n\in\N}\subset\Measure(\Omega)\times\Measure(\Omega_1)\times W^{1,p}(\Omega)$ 
    satisfying
	\begin{enumerate}[label=\textup{(\roman*)}]
		\item\label{it:recfeas}
		 $(\pi_n^*,\mu_{1, n}^*,c_n^*)$ is feasible for \textnormal{(\hyperref[eq:BKgam]{$\text{BK}_n$})} for all $n \in \N$,
		\item\label{it:recconv}
		 $\limsup_{n\to\infty} \JJ_n(\pi_n^*,\mu_{1,n}^*,c_n^*) \leq \JJ(\pi^*,\mu_1^*)$.
	\end{enumerate}
	Then, every weak-$\ast$ accumulation point $(\bar{\pi},\bar{\mu}_1)$ 
	of the sequence of solutions $\{( \bar{\pi}_n, \bar{\mu}_1^n )\}_{n \in \N}$ 
	to the regularized bilevel problem is also a solution to \eqref{eq:BK}.
\end{corollary}

\begin{proof}
    As the feasibility of $(\bar\pi, \bar\mu_1)$   
    has already been established, we only need to prove its 
    optimality, which is a consequence of the existence of 
    a recovery sequence and the weak-$\ast$ lower 
    semicontinuity of $\JJ$.
    To this end, we index the weakly-$\ast$ convergent
    subsequence by the symbol $n$.
    \begin{align*}
        \JJ(\bar{\pi},\bar{\mu}_1)
        &\leq \liminf_{n\to\infty} \JJ(\bar\pi_n,\bar\mu_1^n) \\
        &\leq \liminf_{n\to\infty} \JJ(\bar\pi_n,\bar\mu_1^n) + \frac{1}{p\,\gamma_n} \|\bar c_n - \cost\|_{W^{1,p}(\Omega)}^p\\
        &= \liminf_{n\to\infty} \JJ_n(\bar\pi_n,\bar\mu_1^n,\bar c_n) \\
        &\leq \limsup_{n\to\infty} \JJ_n(\pi_n^*,\mu_{1,n}^*,c_n^*)
		\leq \JJ(\pi^*,\mu_1^*).
	\end{align*}
	The optimality of $(\pi^*,\mu_1^*)$ gives the result.
\end{proof}

The crucial task is now of course to establish the existence of a recovery sequence satisfying 
\ref{it:recfeas} and \ref{it:recconv}. 
So far, unfortunately, we are not able to guarantee the existence of such a sequence in the general setting without further assumptions. 
If, however, $\marg \ll \lambda^{d_2}$, $c(x_1,x_2) = h(x_1-x_2)$ with a strictly convex function $h$, 
and $\JJ$ is even weak-$\ast$ continuous,
then a recovery sequence can be constructed as we will see in the next section.


\subsection{Reverse Approximation in Case of Strictly Convex Costs and an Absolutely Continuous Marginal}
\label{sec:brenier}

\begin{theorem}\label{thm:AbsContStrConvex}
    Suppose that, in addition to our standing assumptions, the following hold true:
 	\begin{enumerate}[label=\textup{(\arabic*)}]
		\item\label{it:Omega12} $\Omega_1 = \Omega_2 =: \Omega_*$ (such that $d_1 = d_2 =: d_*$),
		\item\label{it:abscont} $\marg \ll \lambda_*$ with $\lambda_* := \lambda_1 = \lambda_2$,
		\item\label{it:costh} $\cost(x_1,x_2) = h(x_1-x_2)$ with a function $h: \R^{d_*} \to \R$ that is strictly convex and 
		even symmetric, i.e., $h(-\xi) = h(\xi)$ for all $\xi\in \R^{d_*}$,
		\item\label{it:weakcont} $\JJ$ is upper semicontinuous w.r.t.\ weak-$\ast$ convergence in the first variable, i.e., 
        if $\pi_n \weak^* \pi$ in $\Measure(\Omega)$, then, for every $\mu_1 \in \Measure(\Omega_1)$, there holds
        \begin{equation*}
            \limsup_{n\to\infty} \JJ(\pi_n, \mu_1)
			 \le \JJ(\pi, \mu_1),
        \end{equation*}
        \item The null sequences $\{\gamma_n\}_{n\in \N}$ and $\{\delta_n\}_{n\in\N}$ satisfy \eqref{eq:RegParamSeq}.
    \end{enumerate}   	
    Then, a sequence of solutions $\{(\bar\pi_n, \bar\mu_1^n, \bar c_n)\}_{n \in \N}$ of the regularized 
    bilevel problems \textnormal{(\hyperref[eq:BKgam]{BK$_n$})}, $n \in \N$, 
    converges (up to subsequences) to an optimal solution of \eqref{eq:BK}.
\end{theorem}
\begin{proof}
    The assertion is a direct consequence of Corollary~\ref{cor:ClusterPointOptimal} and 
    the uniqueness of the transport plan under the additional assumptions. 
    Let $(\pi^*, \mu_1^*)$ be an optimal solution of the bilevel problem \eqref{eq:BK}.
    According to the generalized version of Brenier's theorem in \cite[Theorem~2.44]{Vil09}, 
    the additional assumptions~\ref{it:Omega12}--\ref{it:costh}
    ensure the existence of an optimal transport map, which in turn yields a unique optimal transport plan that
    solves the Kantorovich problem 
    with marginals $(\marg, \mu_1^*)$ and cost $\cost(x_1,x_2)$, i.e., 
    \begin{equation*}
        \min\Big\{ \int_\Omega \cost(x_2, x_1)\,\d\varphi(x_1,x_2) : \varphi \in \Pi(\marg, \mu_1^*), \;\varphi \geq 0  \Big\}.
    \end{equation*}        
    For reasons of symmetry, $\pi$ is a solution to the above problem if and only if 
    $\pi'$ defined by $\pi'(B_1\times B_2) = \pi(B_2\times B_1)$, $B_1, B_2\in \BB(\Omega_*)$, solves
    \begin{equation}\tag{KP$^*$}\label{eq:KP*}
        \min\Big\{ \int_\Omega \cost(x_1,x_2)\,\d\varphi(x_1,x_2) : \varphi \in \Pi(\mu_1^*,\marg), \;\varphi \geq 0  \Big\}
    \end{equation}        
    such that the solution set of \eqref{eq:KP*} is a singleton, too. Therefore, $\pi^*$ is the only solution of \eqref{eq:KP*}.
    
    Let us define the sequence
    \begin{equation}\label{eq:recoveryseq}
        \mu_{1,n}^* := \mu_1^*, \quad c_n^* := \cost, \quad 
        \pi_n^* := \SS_n\big( \cost , \TT_1^n(\mu_1^*), \TT_2^n(\marg) \big), \quad n\in \N.
    \end{equation}
    By construction, $(\pi_n^*, \mu^*_{1,n}, c_n^*)$ is feasible for \textnormal{(\hyperref[eq:BKgam]{BK$_n$})} 
    for every $n\in \N$. Since the sequence $\{\pi_n^*\}$ is bounded, there is a weakly-$\ast$ convergent subsequence 
    and, according to Proposition~\ref{prop:WeakLimFeasBilevel}, its limit,
    together with $\mu_1^*$, is feasible for \eqref{eq:BK}, 
    i.e., 
    it is a solution of \eqref{eq:KP*}. Thus, by uniqueness, said weak-$\ast$ limit equals $\pi^*$ and a classical argument 
    by contradiction yields that the whole sequence $(\pi_n^*, \mu^*_{1,n}, c_n^*)$ converges (weakly-$\ast$) to 
    $(\pi^*, \mu^*_{1}, \cost)$. The presupposed weak-$\ast$ continuity of $\JJ$ finally yields that 
    \begin{equation*}
        \limsup_{n\to\infty} \JJ_n(\pi_n^*,\mu_{1,n}^*,c_n^*)
        = \limsup_{n\to\infty} \JJ(\pi_n^*, \mu_1^*)
	\le \JJ(\pi^*, \mu^*_1) \quad \text{as } n \to \infty
    \end{equation*}
    such that $\{(\pi_n^*, \mu^*_{1,n}, c_n^*)\}_{n\in \N}$ is the desired recovery sequence. 
    Corollary~\ref{cor:ClusterPointOptimal} then yields the claim.
\end{proof}

\begin{remark}\label{rem:ohnec}
    An inspection of the proof of Theorem~\ref{thm:AbsContStrConvex} shows that, under the assumptions of this theorem,  
    there is no need for $c$ as additional optimization variable to construct the recovery sequence, 
    as $c_n^*$ is set to $\cost$ in \eqref{eq:recoveryseq}. Accordingly, the assertion of Theorem~\ref{thm:AbsContStrConvex} remains 
    true, if one considers
    \begin{equation}\tag{$\textup{BK}_{\gamma, c_\textup{d}}^\delta$}\label{eq:BKgam2}
    \left\{\quad
    \begin{aligned}
        \inf_{\pi_\gamma,\mu_1} \quad & \JJ(\pi_\gamma, \mu_1)\\
        \text{s.t.} \quad & \mu_1\in \Prob(\Omega_1), \quad \dist(\supp(\mu_1), \partial\Omega_1) \geq \rho, \\
        & \pi_\gamma = \SS_\gamma\big( \cost, \TT_1^\delta(\mu_1), \TT_2^\delta(\marg) \big).
    \end{aligned}
    \right.        
    \end{equation}    
    instead of \eqref{eq:BKgam}. In the 
	finite-dimensional setting, however, it is
	exactly the additional optimization variable $c$, 
    which enables the construction of a recovery
    sequence, see \cite{kant3}.
    
Consequently, one may drop the additional
optimization over $c$ and the term
$\frac{1}{\gamma p}\|c - \cost\|_{W^{1,p}(\Omega)}^p$
in the setting of Theorem \ref{thm:AbsContStrConvex}.
Then, the restrictions on the domain
that ensure the compact embedding
$W^{1,p}(\Omega) \embed C(\Omega)$ are no longer necessary
and one may restrict to $\Omega_1$ and $\Omega_2$
being compact, see also Remark
\ref{rem:BKexist_domain}.
\end{remark}

The additional assumptions in Theorem~\ref{thm:AbsContStrConvex} are certainly rather restrictive, in particular 
condition~\ref{it:weakcont}.
Let us therefore end our considerations by giving two examples fulfilling these assumptions. 

\subsubsection{Marginal Identification Problem}
As an example for a bilevel Kantorovich problem, we 
have already mentioned the problem 
of identifying the marginal $\mu_1$ based on measurements of the transport plan in an observation domain $D\subset \Omega$
in Section \ref{sec:intro}. 
The most natural choice for the upper-level objective probably reads
\begin{equation}\label{eq:noncompobj}
    \JJ(\pi, \mu_1) := \|\pi - \pi^\textup{d} \|_{\Measure(D)}  + \nu \,\| \mu_1  - \mu_1^\textup{d}\|_{\Measure(\Omega_*)},
\end{equation}
where $\pi^{\textup{d}}\in \Measure(D)$ denotes the measurement of the transport plan, 
while $\mu_1^{\textup{d}}\in \Prob(\Omega_*)$ is a guess for the unknown marginal $\mu_1$.
Moreover, $\nu \geq 0$ is a given weighting parameter.
However, an objective of this form does not satisfy condition \ref{it:weakcont} in Theorem~\ref{thm:AbsContStrConvex}. 
To ensure this condition, let us assume that $D$ is an open and bounded domain with a Lipschitz boundary. 
Then, thanks to $p > d$ by our standing assumptions, 
the embedding $W^{1,p}_0(D) \embed C(\overline{D})$ is compact and so,  
by Schauder's theorem, $\Measure(\overline{D})$ embeds compactly in 
$W^{-1,p'}(D) := W^{1,p}_0(D)^*$, where, as usual, $p' = p/(p-1)$ denotes the conjugate exponent.
Therefore, for a given $\pi^{\textup{d}} \in \Measure(\overline{D})$, an objective of the form
\begin{equation}\label{eq:compobj}
    \widetilde\JJ(\pi, \mu_1) := \|\pi - \pi^\textup{d} \|_{W^{-1,p'}(D)}^{p'} 
    + \nu \,\| \mu_1  - \mu_1^\textup{d}\|_{\Measure(\Omega_*)}
\end{equation}
fulfills condition~\ref{it:weakcont}. The $W^{-1,p'}(D)$-norm can be evaluated with the help of the $p$-Laplacian. 
For this purpose, 
denote by $\psi \in W^{1,p}_0(D)$
the unique solution of 
\begin{equation}\label{eq:plaplace}
    -\div (|\nabla \psi|^{p-2} \nabla \psi) = \pi - \pi^{\textup{d}} \quad \text{in } W^{-1,p'}(D).
\end{equation}
Then, it is easily seen that
\begin{equation*}
\begin{aligned}
    \|\pi - \pi^{\textup{d}}\|_{W^{-1,p'}(D)}
    & = \sup_{v\in W^{1,p}_0(D)} \frac{\dual{\pi-\pi^{\textup{d}}}{v}}{\|\nabla v\|_{L^p(D;\R^d)}} \\
    & = \sup_{v\in W^{1,p}_0(D)} \frac{\int_D |\nabla \psi|^{p-2} \nabla \psi \cdot \nabla v\,\d \lambda}{\|\nabla v\|_{L^p(D;\R^d)}}
    = \|\nabla \psi\|_{L^p(D;\R^d)}^{p-1}
\end{aligned}
\end{equation*}
and hence, the objective from \eqref{eq:compobj} becomes $\|\nabla \psi\|_{L^p(D;\R^d)}^p  + \nu \,\| \mu_1  - \mu_1^\textup{d}\|$. 
If one aims to avoid the $p$-Laplace equation, one can resort to an equivalent norm based on the Poisson equation on $D$.
To this end, let $\eta \in \Measure(\overline{D})$ be given and consider
\begin{equation}\label{eq:pdeD}
    \varphi\in W^{1,p'}_0(D), \quad - \laplace \varphi = \eta \quad \text{in } W^{-1,p'}(D),
\end{equation}
where $\dual{-\laplace \varphi}{v} := \int_{D} \nabla \varphi \cdot \nabla v\, \d \lambda$, 
$\varphi \in W^{1,p'}_0(D)$, $v\in W^{1,p}_0(D)$, denotes the Laplace operator. 

\begin{lemma}\label{lem:groeger}
    Let $D$ be a bounded domain of class $C^1$. 
    Then there exists an exponent $p > d$ such that, for every $\eta \in \Measure(\overline{D})$, 
    there exists a unique solution $\varphi \in W^{1,p'}_0(D)$ of \eqref{eq:pdeD}. 
    The associated solution operator denoted by $G: \Measure(\overline{D}) \to W^{1,p'}_0(D)$ is linear and compact.
    
    If $d \leq 3$, the result also applies if $D$ is only
    a bounded Lipschitz domain.
\end{lemma}

\begin{proof}
    According to \cite[Theorem~4.6]{Sim72}, there exists a $p > d$
    such that the Poisson equation admits a unique solution in $W^{1,p}_0(D)$ for right hand 
    sides in $W^{-1,p}(D)$.\footnote{On $C^1$-domains, this even holds for every $p < \infty$, see~\cite[Theorem~4.6]{Sim72}.}
    The same assertion for Lipschitz domains can be found in \cite{Gro89} for the case $d=2$ and \cite{JK95} for $d=3$.
    By duality, there is thus a unique solution in $\varphi \in W^{1,p'}_0(D)$ to the state equation in \eqref{eq:pdeD} 
    for every right hand side in $W^{-1,p'}(D)$. 
    Due to the continuous embedding $W^{1,p}_0(D) \embed C(\overline{D})$
    already mentioned above, there holds $\Measure(\overline{D}) \embeds W^{-1,p'}(D)$ and we obtain 
    the existence and uniqueness of $\varphi \in W^{1,p'}_0(D)$ for every $\pi \in \Measure(\overline{D})$.
    The associated solution operator $G$ is  clearly linear and, by Banach's inverse theorem, continuous.
    The compactness of $G$ follows from the 
    compactness of the embedding $W^{1,p}(D) \embed C(\overline{D})$.
\end{proof}

Unfortunately, the $W^{-1,p'}(D)$-norm cannot be expressed
by means of the solution operator $G$, but there holds
$\|\eta\|_{W^{-1,p'}(D)}  \leq \|\nabla \varphi\|_{L^{p'}(\Omega;\R^d)} \leq \|G\|\, \|\eta\|_{W^{-1,p'}(D)}$
and therefore, 
\begin{equation*}
    \Measure(\overline{D}) \ni \eta \mapsto \|\nabla G \eta \|_{L^{p'}(\Omega;\R^d)} 
\end{equation*}
defines a norm equivalent to the $W^{-1,p'}(D)$-norm. Since $G$ is compact, an objective of the form
\begin{equation}\label{eq:compobj2}
    \widehat\JJ(\pi, \mu_1) := \|\nabla G(\pi - \pi^\textup{d}) \|_{L^{p'}(D;\R^d)}^{p'} 
    + \nu \,\| \mu_1  - \mu_1^\textup{d}\|_{\Measure(\Omega_*)}
\end{equation}
satisfies the condition \ref{it:weakcont} in Theorem~\ref{thm:AbsContStrConvex}.
Thus, if we consider the bilevel Kantorovich problem \eqref{eq:BK} with $\widetilde\JJ$ or $\widehat{\JJ}$, 
Theorem~\ref{thm:AbsContStrConvex} applies. Let us shortly turn to the associated regularized problems.
For this purpose, recall the following result from \cite{LMM21}:

\begin{lemma}[{\cite[Theorem~2.11]{LMM21}}]\label{lem:quadreg}
    Consider the regularized Kantorovich problem \eqref{eq:KPgam} with marginals $\mu_i \in L^2(\Omega_i)$, $i=1,2$, 
    and a cost function $c\in L^2(\Omega)$. 
    Assume that the marginals satisfy $\|\mu_1\|_{L^1(\Omega_1)} = \|\mu_2\|_{L^1(\Omega_2)}$
    and $\mu_i \geq \delta$ $\lambda_i$-a.e.\ in $\Omega_i$, $i=1,2$, with a constant $\delta > 0$. 
    Moreover, assume that there exists a constant $\cbound > -\infty$ such that 
    $c \geq \cbound$ $\lambda$-a.e.\ in $\Omega$. Then, $\pi_\gamma\in L^2(\Omega)$ is a solution of \eqref{eq:KPgam}
    if and only if there exist functions $\alpha_1\in L^2(\Omega_1)$ and $\alpha_2\in L^2(\Omega_2)$ satisfying	
    \begin{subequations}\label{eq:KantProbL2OptSystem}
	\begin{alignat}{3}
	    \pi_\gamma - \frac{1}{\gamma} (\alpha_1 \oplus \alpha_2 - c)_+ & = 0 & \quad & \lambda\text{-a.e. in}~\Omega,
	    \label{eq:KantProbL2OptSystem-a} \\
        \int_{\Omega_2} \pi_\gamma(x_1, x_2) \d\lambda_2(x_2) &= \mu_1(x_1) & & \lambda_1\text{-a.e. in}~\Omega_1, 
        \label{eq:KantProbL2OptSystem-b}\\
        \int_{\Omega_1}\pi_\gamma(x_1, x_2)\d\lambda_1(x_1) &= \mu_2(x_2) & & \lambda_2\text{-a.e. in}~\Omega_2.
        \label{eq:KantProbL2OptSystem-c}
    \end{alignat}
	\end{subequations}
    Herein, $(\alpha_1 \oplus \alpha_2)(x_1,x_2)\coloneqq\alpha_1(x_1)+\alpha_2(x_2)$ $\lambda$-a.e.\ in $\Omega$
	refers to the direct sum of $\alpha_1 \in L^2(\Omega_1)$ and $\alpha_2\in L^2(\Omega_2)$, while, for given 
	$u\in L^2(\Omega)$, $(u)_+(x) \coloneqq \max\{u(x);0\}$ $\lambda$-a.e.\ in $\Omega$ denotes the pointwise 
	maximum. 
\end{lemma}

With this result at hand, we can erase the regularized optimal transport plan $\pi_\gamma$ from \eqref{eq:BKgam2}
by using the necessary and sufficient conditions in \eqref{eq:KantProbL2OptSystem}.
Let us exemplarily consider \eqref{eq:BKgam2} with the objective from \eqref{eq:compobj2}. 
The corresponding regularized bilevel problem then reads as follows:
\begin{equation}\label{eq:BKgam_marg}
    \left\{\;
    \begin{aligned}
        \min \quad &
        \|\nabla \varphi \|_{L^{p'}(D;\R^d)}^{p'} 
        + \nu \,\| \mu_1  - \mu_1^\textup{d}\|_{\Measure(\Omega_*)}\\
        \text{s.t.} \quad & \alpha_1, \alpha_2\in L^2(\Omega_*), \quad \mu_1 \in \Prob(\Omega_*), \quad 
        \varphi \in W^{1,p'}_0(D),\\
        & \dist(\supp(\mu_1), \partial\Omega_*) \geq \rho, \\
        & - \laplace \varphi = \tfrac{1}{\gamma} (\alpha_1 \oplus \alpha_2 - \cost)_+ - \pi^{\textup{d}} \quad \text{in } W^{-1,p'}(D),\\
        & \int_{\Omega_*} (\alpha_1 \oplus \alpha_2 - \cost)_+(x_1, x_2)\d\lambda_*(x_2) = \gamma \, \TT_1^\delta(\mu_1)(x_1) 
        \quad \text{a.e.\ in } \Omega_*,\\
        & \int_{\Omega_*} (\alpha_1 \oplus \alpha_2 - \cost)_+(x_1, x_2)\d\lambda_*(x_1) = \gamma \, \TT_2^\delta(\marg)(x_2) 
        \quad \text{a.e.\ in } \Omega_*.
    \end{aligned}
    \right.
\end{equation}

\begin{remark}\label{rem:yosida}
    If one replaces the lower-level Kantorovich problem as
    in \eqref{eq:BK} by its necessary and sufficient 
    (and thus equivalent) optimality conditions, then an optimization problem with complementarity constraints (MPCC) in 
    $\Measure(\Omega)$ is obtained, see \cite[Section~3]{kant1}. Problems of this type are challenging, since standard 
    constraint qualifications are typically violated, even in finite dimensions. For this reason, regularization 
    and relaxation techniques are frequently applied, we only refer to \cite{HKS13} and the references therein.
    In light of the above reformulation based on Lemma~\ref{lem:quadreg}, the quadratic regularization of the Kantorovich 
    problem can be interpreted as a relaxation of the complementarity constraints, too, since it is 
    precisely the 
    Moreau--Yosida regularization of the dual Kantorovich problem as follows the considerations in \cite[Section~2.2]{LMM21}. 
    Indeed the Moreau--Yosida regularization of inequality constraints relaxes the complementarity constraints, 
    but the associated regularized optimization problems
    typically still contain a (moderately) non-smooth term, which is also observed here, in form of the $\max$-function 
    involved in \eqref{eq:BKgam_marg}. Even though it complicates the (numerical) solution of \eqref{eq:BKgam_marg}, 
    this is a desirable feature, as the $\max$-function promotes the sparsity of the optimal transport plan, 
    cf.\ the numerical experiments in \cite[Section~4.3]{LMM21}. A similar observation is also made in context of optimal 
    control of VIs such as the obstacle problem, where the
    $\max$-operator that arises from the Moreau--Yosida regularization of 
    the complementarity system can be further smoothed (see e.g.\ \cite{SW13}) or tackled by a semi-smooth Newton method 
    (cf.\ \cite{CMWC18}).
\end{remark}

Concerning the convergence of solutions to \eqref{eq:BKgam_marg}
for regularization parameters tending to zero, 
Theorem~\ref{thm:AbsContStrConvex} and Remark~\ref{rem:ohnec} imply the following result:

\begin{corollary}
    Let $D\subset \Omega$ satisfy the assumptions from Lemma~\ref{lem:groeger} and let $\pi^{\textup{d}} \in \Measure(\overline{D})$ 
    be given. Moreover, assume that $\marg \in L^1(\Omega_*)$
    and that the transportation costs $\cost$ fulfill 
    condition~\ref{it:costh} of Theorem~\ref{thm:AbsContStrConvex}. Then, for every
    sequence $\{(\gamma_n, \delta_n)\}_{n\in \N}$ tending to zero and fulfilling \eqref{eq:RegParamSeq}, 
    there exists a subsequence of solutions 
    to \eqref{eq:BKgam_marg} denoted by $(\bar\mu_1^n,  \bar\alpha_1^n, \bar \alpha_2^n, \bar \varphi_n)$ such that 
    \begin{equation*}
    \begin{aligned}
        \bar\mu_1^n \weak^* \bar\mu_1 \text{ in } \Measure(\Omega_*), \enskip 
        \tfrac{1}{\gamma}
        (\bar\alpha_1^n \oplus \bar\alpha_2^n - \cost)_+  \weak^* \bar\pi\text{ in } \Measure(\Omega), \enskip 
        \bar \varphi_n & \to \bar \varphi \text{ in } W^{1,p'}_0(D)
    \end{aligned}            
    \end{equation*}
    and the limit $(\bar\mu_1, \bar\pi, \bar \varphi)$ is a 
    solution to the bilevel Kantorovich problem with
    the objective from
    \eqref{eq:compobj2}.
\end{corollary}

Though the transport plan has been erased from the regularized bilevel problem, the quadratic regularization 
does not completely resolve the ``curse of dimensionality'' associated with the Kantorovich problem, since it involves a 
PDE on $D$, which is a subset of $\Omega$ and thus 
$\dim(D) = d_1 + d_2$.
By contrast, the next example provides a substantial reduction of the dimension for the price however that 
one looses convexity.

\subsubsection{Optimal Control in Wasserstein Spaces}
For our second example for a problem fulfilling the assumptions of Theorem~\ref{thm:AbsContStrConvex}, 
we suppose that $\Omega_* := \Omega_1 = \Omega_2$ satisfies the same assumptions as 
$\Omega = \Omega_* \times \Omega_*$, i.e., it coincides with the closure of its interior 
and has a Lipschitz boundary.
As in case of $\Omega$, we write
$W^{1,q}_0(\Omega_*) := W^{1,q}_0(\interior(\Omega_*))$
for the Sobolev space with vanishing trace. 
In order to have Lemma~\ref{lem:groeger} at our disposal, we moreover assume that $d_* \leq 3$.
Furthermore, we fix $\nu > 0$, $\beta > 1$, and $q>d_*$ and set $q' = q/(q-1)$. 
We then consider the following \emph{elliptic optimal control problem}:
\begin{equation}\tag{OCP}\label{eq:ocp}
    \left\{\quad
    \begin{aligned}
        \min_{y,\mu_1} \quad & \tfrac{1}{2}\,\| y - y_\textup{d}\|_{L^2(\Omega_*)}^2 + \nu \, W_\beta(\mu_1, \marg)^\beta \\
        \text{s.t.} \quad &  y \in W^{1,q'}_0(\Omega_*), \quad \mu_1 \in \Prob(\Omega_*), \\
        & \dist(\supp(\mu_1), \partial\Omega_*) \geq \rho, \quad - \laplace y =  \mu_1 \;\text{ in }W^{-1,q'}(\Omega_*),
    \end{aligned}
    \right.
\end{equation}
where $W_\beta(\mu_1, \marg)$ denotes the Wasserstein-$\beta$-distance between $\mu_1$ 
and $\marg$ given by
\begin{equation}\label{eq:wasserstein}
    W_\beta(\mu_1, \marg) := \min\left\{\int_\Omega  |x_1 - x_2|^\beta\,\d\varphi(x_1,x_2)
    \colon\varphi\in\Pi(\mu_1,\marg), \,\varphi\geq 0\right\}^{\frac{1}{\beta}}.
\end{equation}
Moreover, $y_{\textup{d}} \in L^2(\Omega_*)$ is a given desired state. 
Note that, again, due to $q > d_*$, there holds $\Measure(\Omega_*) \embed W^{-1,q'}(\Omega_*)$ 
with compact embedding such that the right hand side in Poisson equation in \eqref{eq:ocp} is well defined.

\begin{remark}
    Depending on the application background, it might be favorable to 
    measure the distance of the control $\mu_1$ to a given prior $\marg$ in the Wasserstein distance 
    instead of taking, e.g., the total variation norm $|\mu_1 - \marg|(\Omega_*)$. 
    Optimal control problems in measure spaces with the total variation as control costs have intensively been studied 
    in literature, we only refer to \cite{CCK12, CK14, CK19} and the references therein.    
    However, the total variation might be a too strong norm for several applications, see, e.g.,
    \cite{B2018, PZ2019, TPG2016}
    for beneficial properties of the Wasserstein distance.
\end{remark}

Due to $d_* \leq 3$ and our regularity assumptions on $\Omega_*$, Lemma~\ref{lem:groeger} is applicable 
and guarantees the existence of a unique solution $y$ of the Poisson equation in \eqref{eq:ocp} for every 
$\mu_1 \in \Measure(\Omega_*)$. The associated solution operator is again denoted by 
$G: \Measure(\Omega_*) \to W^{1,q'}_0(\Omega_*)$.
Given this solution operator, we can erase the state $y$ from \eqref{eq:ocp}, which, 
together with \eqref{eq:wasserstein}, leads to the following reformulation:
\begin{equation*}
\begin{aligned}
    \eqref{eq:ocp} \; & \Longleftrightarrow \;
    \left\{\;
    \begin{aligned}
        \min_{\pi,\mu_1} \quad & \JJ(\pi, \mu_1):= 
        \tfrac{1}{2}\,\| G \mu_1 - y_\textup{d}\|_{L^2(\Omega_*)}^2 + \nu \int_\Omega |x_1 - x_2|^\beta \,\d\pi(x_1,x_2) \\
        \text{s.t.} \quad & \mu_1\in \Prob(\Omega_*), \quad \dist(\supp(\mu_1), \partial\Omega_*) \geq \rho, \\
        & \pi \in \argmin\left\{\int_\Omega  |x_1 - x_2|^\beta\,\d\varphi(x_1,x_2)
        \colon\varphi\in\Pi(\mu_1,\marg), \,\varphi\geq 0\right\}
    \end{aligned}
    \right. \\
    & \Longleftrightarrow \;
    \left\{\;
    \begin{aligned}
        \min_{\pi} \quad & 
        \tfrac{1}{2}\,\| G ({P_1}_\# \pi) - y_\textup{d}\|_{L^2(\Omega_*)}^2 + \nu \int_\Omega |x_1 - x_2|^\beta \,\d\pi(x_1,x_2) \\
        \text{s.t.} \quad & \pi \in \Measure(\Omega),\quad
        \dist(\supp({P_1}_\# \pi), \partial\Omega_*) \geq \rho,\\
        & \pi \in\Pi({P_1}_\# \pi,\marg), \quad\pi\geq 0.
    \end{aligned}
    \right.
\end{aligned}
\end{equation*}
While the first reformulation is a bilevel problem of the form \eqref{eq:BK}, the second
one is (astonishingly) a convex problem, which is of course a favorable feature. However, as the transport plan is the
optimization variable, we have to deal with a problem in $\Omega = \Omega_* \times \Omega_*$.
Using Lemma~\ref{lem:quadreg}, the quadratic regularization allows to avoid this ``curse of dimensionality''.
Abbreviating the transportation costs associated with the Wasserstein-$\beta$-distance 
by $\cost$, i.e., $\cost(x_1,x_2) = |x_1-x_2|^\beta$, the regularized counterpart to \eqref{eq:ocp} reads:
\begin{equation}\tag{OCP$_\gamma^\delta$}\label{eq:ocpgam}
    \left\{\;
    \begin{aligned}
        \min \quad &
        \begin{aligned}[t]
            \tfrac{1}{2}\,\| y - y_\textup{d}\|_{L^2(\Omega_*)}^2 
            &+ \frac{\nu}{\gamma} 
            \int_\Omega \cost\, (\alpha_1 \oplus \alpha_2 - \cost)_+\,\d\lambda \\
        \end{aligned}\\       
        \text{s.t.} \quad &  y \in W^{1,q'}_0(\Omega_*), \;\; \alpha_1, \alpha_2\in L^2(\Omega_*), \; \;\mu_1 \in \Prob(\Omega_*), \\
        &  \dist(\supp(\mu_1), \partial\Omega_*) \geq \rho, \quad - \laplace y =  \mu_1 \;\text{ in }W^{-1,q'}(\Omega_*),\\
        & \int_{\Omega_*} (\alpha_1 \oplus \alpha_2 - \cost)_+(x_1, x_2)\d\lambda_*(x_2) = \gamma \, \TT_1^\delta(\mu_1)(x_1) 
        \;\; \text{a.e.\ in } \Omega_*,\\
        & \int_{\Omega_*} (\alpha_1 \oplus \alpha_2 - \cost)_+(x_1, x_2)\d\lambda_*(x_1) = \gamma \, \TT_2^\delta(\marg)(x_2) 
        \;\;\text{a.e.\ in } \Omega_*.
    \end{aligned}
    \right.
\end{equation}
We observe that this problem no longer contains any variable or constraint in $\Omega$, 
but only quantities and constraints in $\Omega_*$. However, the price we have to pay for this reduction of the dimension is 
a loss of convexity, since \eqref{eq:ocpgam} is no longer a convex problem due to the equality constraints including 
the $\max$-function.

\begin{remark}
    We point out that alternative regularization procedures like the entropic regularization 
    lead to similar non-convex equality constraints, see \cite[Theorem~4.8]{CLM21}.
    Alternatively, one might replace the Kantorovich problem in the bilevel formulation of \eqref{eq:ocp}
    by its dual problem without any further regularization. At first glance, this seems to be promising, 
    since the dual Kantorovich problem is posed in $C(\Omega_*) \times C(\Omega_*)$ instead of 
    $\Measure(\Omega_* \times \Omega_*)$ indicating the desired reduction of the dimension, 
    see \cite[Theorem~5.10]{Vil09} for the derivation of the dual Kantorovich problem.
    However, the bilevel problem then becomes a $\min$-$\max$-problem including a constraint in $\Omega_* \times \Omega_*$.
    To summarize, it seems that a reduction of the dimension without increasing the complexity of the problem 
    is impossible.
\end{remark}

Since the objective $\JJ$ in the bilevel formulation of \eqref{eq:ocp} is linear in $\pi$, Theorem~\ref{thm:AbsContStrConvex}
is applicable, which yields the following

\begin{corollary}\label{cor:ocp}
    Suppose that $\Omega_* \in \R^{d_*}$, $d_*\leq 3$, is such that $\overline{\interior(\Omega_*)} = \Omega_*$ and 
    $\interior(\Omega_*)$ is a bounded Lipschitz domain.
    Let $\nu > 0$ and $\beta > 1$ be given and let $q>d_*$ be the exponent from Lemma~\ref{lem:groeger}.
    Moreover, assume that $\marg \in L^2(\Omega_*)$.
    Then, for a given a sequence
    $\{(\gamma_n, \delta_n)\}_{n\in \N}$ 
    tending to zero and fulfilling \eqref{eq:RegParamSeq}, 
    there exists a subsequence of solutions \eqref{eq:ocpgam} denoted by 
    $(\bar\mu_1^n,  \bar\alpha_1^n, \bar \alpha_2^n, \bar y_n)$
    such that
    \begin{equation*}
    \begin{aligned}
        \bar\mu_1^n \weak^* \bar\mu_1 \;\; \text{in } \Measure(\Omega_*), \;\;
        \tfrac{1}{\gamma} (\bar\alpha_1^n \oplus \bar\alpha_2^n
         - \cost)_+ \weak^* \bar\pi \;\;\text{in } \Measure(\Omega), \;\;
        \bar y_n \to \bar y \;\; \text{in } W^{1,q}_0(\Omega_*)
     \end{aligned}            
     \end{equation*}
    and the limit $(\bar\mu_1, \bar y)$ is a solution of \eqref{eq:ocp}.
\end{corollary}

\begin{proof}
    We verify the assumptions of
    Theorem~\ref{thm:AbsContStrConvex} for the reduced form of 
    \eqref{eq:ocp}. 
    First, $\cost (x_1, x_2) = h(x_1-x_2) := |x_1 - x_2|^\beta$ 
    is clearly symmetric and strictly convex (since $\beta > 1$ 
    by assumption) such that condition~\ref{it:costh} is met. 
    Moreover, the upper-level objective is even weak-$\ast$ 
    continuous due to the linearity of 
    $\pi \mapsto \int_\Omega \cost \d\pi$ and the compactness of $G$. Moreover, it is clearly bounded on bounded 
    sets and therefore, $\JJ$ fulfills our standing assumptions as well as 
    condition~\ref{it:weakcont}. Thus, since $\marg \ll \lambda_*$ by 
    assumption, all hypotheses of Theorem~\ref{thm:AbsContStrConvex} are fulfilled and the assertion follows by Remark~\ref{rem:ohnec}.
\end{proof}

In particular, due to the non-smooth $\max$-operator in its
constraints, \eqref{eq:ocpgam} itself is a challenging problem. 
However, for instance by employing a further smoothing of $\max$, 
it might be possible to approximate \eqref{eq:ocpgam} by smooth infinite-dimensional optimization problems, 
similarly to optimal control problems governed by VIs, see Remark~\ref{rem:yosida}.
In light of Corollary~\ref{cor:ocp}, this might open a way for a numerical solution of \eqref{eq:ocp} without ``falling victim 
to the curse of dimensionality''. 
An efficient solution of \eqref{eq:ocpgam} would however go beyond the scope of this paper and is subject to future research.


\bibliographystyle{plain}
\bibliography{ref}

\end{document}